\documentclass[11pt]{article}

\usepackage{fixed_locus}

\title{Compressive Domains and a Bound for the Number of
  Components of the Fixed Locus of a Self-Map of the Berkovich Line}

\author{Xander Faber \\
  IDA / Center for Computing Sciences \\
Bowie, MD \\
xander@super.org
\and
Niladri Patra \\
Indian Statistical Institute, Delhi Centre, \\
S. J. S. Sansanwal Marg, New Delhi, India. \\
niladri@math.tifr.res.in}

\begin{document}

\maketitle

\begin{abstract}
  We introduce the notion of a ``compressive domain'' for the action of a
  rational function on the Berkovich projective line over a complete
  nontrivially-valued algebraically closed nonarchimedean field. We prove that
  such a domain always contains a classical fixed point, and we leverage this
  fact to give a sharp upper bound for the number of connected components of
  the fixed locus of a rational function. We give a second proof for polynomial
  functions that uses a previously unpublished mass formula of
  Rivera-Letelier. Finally, we give an explicit formula for the crucial weight
  inside a compressive domain as a function of the number of classical fixed
  points and boundary points.
\end{abstract}



\section{Introduction}

Let $K$ be an algebraically closed field that is complete with respect to a
nontrivial nonarchimedean absolute value. Write $p \ge 0$ for the residue
characteristic of $K$. We continue our study of the locus of points of the
analytic projective line $\BerkK$ that are fixed under the action of a
nonconstant rational function $f \in K(z)$ of degree $d \ge 1$. Write
\[
   \Fix(f) = \{x \in \BerkK : f(x) = x\}.
\]
In the previous article \cite{Faber_Patra_fixed_structure}, we classified the
connected components of $\Fix(f)$ according to whether they contained classical
($K$-rational) fixed points or repelling type~II fixed points, and we proved a
number of results on the local and global structure of these components. A
study of the ``hyperbolic components'' of $\Fix(f)$ --- i.e., those components
without classical fixed points --- allowed us to conclude that $\Fix(f)$ has at
most $d+1$ connected components if $p = 0$ or $p > d$, and that $\Fix(f)$ has
at most $2d$ in the remaining cases \cite[Thm.~B]{Faber_Patra_fixed_structure}.

The primary goal for this article is to prove the sharp upper bound $d +
\lfloor d/p \rfloor + 1$ for the number of components of the fixed locus when
$0 < p \le d$. To that end, we introduce the notion of ``compressive domain''
for $f$, which is a simple domain on which $f$ is locally compressive near the
fixed boundary points. Our high-level strategy is then relatively simple:
\begin{enumerate}
  \item Prove an existence theorem for classical fixed points inside
    compressive domains (Theorem~C);
  \item Show that there is a lower bound for the number of distinct classical
    fixed points in a compressive domain $U$ in terms of the residue
    characteristic and the number of hyperbolic components in $U$
    (Proposition~\ref{prop:compressive_classical_bound}); and
  \item Conclude the proof using the standard upper bound $d+1$ for the number of
    distinct classical fixed points of a rational function of degree~$d$.
\end{enumerate}
We give a second proof of the bound for the number of components in the case of
polynomial functions, and we also exhibit polynomials of every degree~$d$ that
achieve our upper bound.

Our second goal for this article is to produce an explicit formula for Rumely's
crucial weight in a compressive domain (Theorem~E). This is a substantial
generalization of Rumely's weight formula for $\BerkK$; see
\cite[Thm.~1.3]{Faber_Patra_fixed_structure} or
\cite[Thm.~6.1]{Rumely_new_equivariant}.


\subsection{Results}

We now give a more detailed description of the main results of this paper. 

Recall that a \textbf{simple domain} is the intersection of finitely many open
disks $B_x(\vv)^- \subset \BerkK$ with type II or type III boundary points.

\begin{definition}
  \label{def:compressive}
  Let $f \in K(z)$ be a nonconstant rational function, and let $U =
  \bigcap_{i=1}^n B_{x_i}(\vv_i)^-$ be a simple domain in $\BerkK$ with
  $\partial U = \{x_1, \ldots, x_n\}$. We say that \textbf{$U$ is a compressive
    domain (for $f$)} if the following hold:
\begin{itemize}
  \item $f(x_i) = x_i$ for $i = 1, \ldots, n$; and
  \item $\vv_i \in T_{x_i}$ is a critically fixed direction for $T_{x_i} f$ for
    $i = 1, \ldots, n$.
\end{itemize}
\end{definition}

The assumption that $\partial U = \{x_1, \ldots, x_n\}$ implies that no disk
$B_{x_i}(\vv_i)^-$ is contained in any other, so that our presentation of the
simple domain is minimal. Each inward pointing direction $\vv_i$ is critically
fixed, which is equivalent to saying that $\vv_i$ is a repelling direction for
$f$. Consequently, we can think of $f$ as ``applying inward pressure'' to $U$
from its fixed points $x_1, \ldots, x_n$. Note that we allow the degenerate
case $U = \BerkK$ ($n=0$), as all of our results hold in this setting.

With this terminology, our classical fixed point theorem is easy to state:

\begin{theoremC}
  Let $f$ be a nonconstant function. A compressive domain for $f$ contains a
  classical fixed point.
\end{theoremC}

For the sake of intuition, consider a continuous map $g : B \to \RR^n$, where
$B$ is the closed unit ball. Brouwer's fixed point theorem asserts that if
$g(B) \subset B$, then $g$ admits a fixed point in $B$
\cite[Cor.~2.15]{Hatcher_AT}. If instead we only know that $g(\partial B)
\subset B$, where $\partial B$ is the boundary sphere of $B$, then composing
Brouwer's result with a retraction-to-the-sphere argument again produces a
fixed point in $B$. Theorem~C is more like the latter result in that we are
imposing a boundary behavior condition rather than a self-mapping
condition. Indeed, when our compressive domain $U$ has at least two boundary
points, one can show that $f(U) = \BerkK$. 

Regarding the novelty of Theorem~C, when $U$ is an open disk with $f(U) \subset
U$, our fixed-point result is an elementary consequence of a well known result
on attracting fixed points from nonarchimedean dynamics
\cite[Thm.~10.86]{Baker-Rumely_BerkBook_2010}. When $U$ is an open disk with
$f(U) = \BerkK$, Theorem~C follows immediately from Rumely's First
Identification Lemma \cite[Lem.~2.1]{Rumely_new_equivariant}. The case where
$U$ has at least two boundary points appears to be entirely new to the
literature.

\begin{theoremD}
  Let $f \in K(z)$ be a rational function of degree $d \geq 2$, and let $p$ be
  the residue characteristic of $K$. If $0 < p \le d$, then the number of
  connected components of the fixed locus of $f$ is at most
  \[
     d + \left\lfloor \frac{d}{p}\right\rfloor + 1,
     \]
  where $\lfloor \alpha \rfloor$ denotes the largest integer less than or equal
  to $\alpha$.
\end{theoremD}

The second author previously showed that the stronger bound $d+2$ can be
obtained when $f$ has potential good reduction
\cite{Patra_fixed}. Rivera-Letelier suggested to the authors that the number of
components of the fixed locus of a polynomial map could be bounded by applying
an unpublished mass formula for repelling fixed points. We state and prove his
mass formula in Section~\ref{sec:polynomials}, and we recover the bound in
Theorem~D for polynomials. The proof sheds additional light on the structure of
the fixed locus for polynomials.  In Section~\ref{sec:examples}, we give
polynomial examples to show that the bounds in Theorem~B of the previous paper
and Theorem~D above are sharp for every nonarchimedean field $K$ and every
degree~$d$. The authors were surprised to find that polynomial maps are
flexible enough to exhibit the behaviors necessary to achieve these bounds in
all cases.

Next, we go further in our analysis of compressive domains to give an explicit
formula for how much of Rumely's crucial weight lives inside it. (See
\cite{Rumely_new_equivariant} or \cite[\S8]{Faber_Patra_fixed_structure} for a
definition of the crucial weight.)

\begin{theoremE}
  \label{prop:compressed_crucial_weight}
  Let $f$ be a nonconstant rational function and $U$ a compressive
  domain. Suppose that $U$ has $n$ boundary points and contains $\bar c$
  classical fixed points for $f$, counted according to fixed-point
  multiplicity. Writing $\wR(x)$ for the crucial weight of a point $x$, we have
  \[
  \sum_{x \in U} \wR(x) = n + \bar c - 2.
  \]
\end{theoremE}

In the case where $U = \BerkK$, this is precisely Rumely's Weight Formula
\cite[Thm.~1.3]{Faber_Patra_fixed_structure}. Our proof of Theorem~E uses a
careful analysis of the combinatorial structure of the boundary points of $U$,
the classical fixed points, and the repelling type~II fixed points.  Since
Rumely deduces his weight theorem using potential theory, we assume there is
another proof of Theorem~E that proceeds along those lines.

Theorem~E admits an interesting connection to two dynamical equidistribution
theorems. Recall that there is a canonical probability measure $\mu_f$ on
$\BerkK$ determined by the two conditions (1) $f^* \mu_f = d\, \mu_f$, and (2)
$\mu_f$ does not charge classical points of $\BerkK$. (See, for example,
\cite[Thm.~10.2]{Baker-Rumely_BerkBook_2010},
\cite[\S4]{Chambert-Loir_Measures_2005}, or
\cite[Thm.~A]{Favre_Rivera-Letelier_Ergodic_2010}.) We can approximate the
measure $\mu_f(U)$ in two different ways. For simplicity, assume that $f$ does
not have potential good reduction, so that $\mu_f$ does not charge any point of
$\BerkK$.

Write $f^m$ for the $m$-fold iterate of $f$. At least when $K$ has
characteristic zero, a consequence of
\cite[Thm.~B]{Favre_Rivera-Letelier_Ergodic_2010} is that
\begin{equation}
  \label{eq:canonical_equidistribution}
   \lim_{m \to \infty} \frac{1}{d^m + 1} \sum_{\substack{f^m(x) = x \\ x \in
       U(K)}} 1 = \mu_f(U),
\end{equation}
where we count the periodic points on the left side according to their
multiplicity as a solution to the equation $f^m(z) = z$. Intuitively, the
proportion of periodic points of period~$m$ inside $U$ converges to the
canonical measure of $U$.

Similarly, a consequence of \cite[Thm.~2]{Jacobs_crucial_equidistribution} is that
\begin{equation}
  \label{eq:crucial_equidistribution}
   \lim_{m \to \infty} \frac{1}{d^m - 1} \sum_{x \in U} w_{\mathrm{R}, f^m}(x) = \mu_f(U),
\end{equation}
where $w_{\mathrm{R}, f^m}(U)$ is the crucial weight of the set $U$ associated
to the iterate $f^m$. Equivalently, the proportion of crucial points for $f^m$
in $U$ converges to the canonical measure of $U$. Comparing
\eqref{eq:canonical_equidistribution} and \eqref{eq:crucial_equidistribution},
we see that
\begin{equation}
  \label{eq:crucial_classical}
\sum_{x \in U} w_{\mathrm{R}, f^m}(x) =
\sum_{\substack{f^m(x) = x \\ x \in U(K)}} 1 + o(d^m) \qquad (m \to \infty).
\end{equation}
Noting that $U$ is compressive for $f^m$ if it is compressive for $f$, Theorem~E asserts that
\[
\sum_{x \in U} w_{\mathrm{R}, f^m}(U) =
n  + \left(\sum_{\substack{f^m(x) = x \\ x \in U(K)}} 1\right) - 2 \qquad (m \to \infty).  
\]
That is, the $o(d^m)$ error term in \eqref{eq:crucial_classical} is the uniform
constant $n - 2$, where $n$ is the number of boundary points for $U$.
 


\subsection{Notation and conventions}

We keep the notation and conventions of the prequel
\cite[\S1.2]{Faber_Patra_fixed_structure}. We write $p$ for the residue
characteristic of $K$. If $S$ is a set of points in $\BerkK$, we write
$\partial S$ for its topological boundary. The \textbf{crucial weight} of $S$
is the sum of the crucial weight of the points in $S$
\cite[\S8]{Faber_Patra_fixed_structure}.

If $x$ is a type~II or type~III fixed point for $f$, we write $\Ncf(f,x)$ for
the number of critically fixed directions for the tangent map $T_x f$. Note
that $\Ncf(f,x) = 0$ when $x$ is of type~III.  Similarly, we define $\Nshear
(f,x)$ to be the number of shearing directions at a type~II or type~III point
$x \in \BerkK$. Again, $\Nshear(f,x) = 0$ if $x$ is of type~III. For $X \subset
\BerkK$, we define
\[
\Ncf(f,X) = \sum_{x \in X} \Ncf(f,x) \quad \text{and} \quad
\Nshear(f,X) = \sum_{x \in X} \Nshear(f,x).
\]


\section{A fixed point theorem for simple domains}
\label{sec:compressive}

Our goal for this section is to prove Theorem~C. 

Let $U$ be a compressive domain for $f$, where $U = \bigcap B_{x_i}(\vv_i)^-$
as in Definition~\ref{def:compressive}.  Using Rumely's First Identification
Lemma, we can count the number of classical fixed points inside $U$ in terms of
the surplus multiplicities of the disks $B_{x_i}(\vv_i)^-$. This is the content
of the next result. Ideally, a proof of Theorem~C would argue that our explicit
count is positive; we were unable to make progress on this
approach. Nevertheless, the count is interesting in its own right.

\begin{proposition}
  \label{prop:compressive_count}
  Let $f$ be a rational function of degree $d \ge 1$, and let $U =
  \bigcap_{i=1}^n B_{x_i}(\vv_i)^-$ be a compressive domain for $f$. Write
  $s_f(\vv_i)$ for the surplus multiplicity of the direction $\vv_i$. The
  number of classical fixed points in $U$ is
  \[
     1 + \sum_{i=1}^n s_f(\vv_i) - d(n-1).
  \]
\end{proposition}
    
\begin{proof}
   If $n = 0$, then $U = \BerkK$ and $f$ admits $d+1$ classical fixed
   points. In what remains, we assume that $n \ge 1$.
        
   Fix an index $i \in \{1, \ldots, n\}$. Since $x_i$ is fixed and $\vv_i$ is
   not a multiple fixed point for $T_{x_i} f$, Rumely's First Identification
   Lemma shows that $B_{x_i}(\vv_i)^-$ contains $1 + s_f(\vv_i)$ classical
   fixed points. Thus, the number of classical fixed points in the complement
   $\BerkK \smallsetminus B_{x_i}(\vv_i)^-$ is
   \[
     (d+1) - (1 + s_f(\vv_i)) = d - s_f(\vv_i).
   \]
   Summing over all indices $i$, we find that $\BerkK \smallsetminus U$
   contains $\sum_{i=1}^{n} (d - s_f(\vv_i))$ classical fixed points. We
   conclude that the number of classical fixed points in $U$ is
   \[
       (d+1) - \sum_{i=1}^n \big(d - s_f(\vv_i)\big) = 1 + \sum_{i=1}^n
   s_f(\vv_i) - d(n-1). \qedhere
   \]
\end{proof}

We say that a compressive domain $U$ is \textbf{minimal} if any compressive
domain $V \subset U$ for $f$ satisfies $V = U$.
    
\begin{lemma}
  \label{lem:fixed_implies_classical}
  Let $U$ be a minimal compressive domain for a rational function $f$. If $U$
  contains a fixed point for $f$, then it contains a classical fixed point for
  $f$.
\end{lemma}

\begin{proof}
  Suppose that $U$ contains no classical fixed point for $f$. By hypothesis,
  there is a fixed point in $U$; let $X$ be the component of the fixed locus of
  $f$ containing this point. Since $f$ is locally repelling at $x_i$ in the
  direction pointing into $U$, the component $X$ is contained entirely inside
  $U$. As $X$ contains no classical fixed point, it is a hyperbolic component.
        
  Let $x \in X$ be a repelling fixed point. Suppose that $\vv \in T_x$ is a
  critically fixed direction under $T_x f$. If $B_{x}(\vv)^{-}$ contains a
  boundary point of $U$, then $V$ is a compressive domain for $f$, where $V$ is
  the component of $U \smallsetminus \{x\}$ contained in $B_{x}(\vv)^{-}$. This
  contradicts the minimality of $U$. Thus, $B_x(\vv)^-$ does not contain a
  boundary point of $U$, which forces $B_x(\vv)^- \subset U$. Rumely's First
  Identification Lemma implies that a classical fixed point of $f$ lies in
  $B_x(\vv)^-$, which again is a contradiction. We conclude that there is no
  critically fixed direction at $x$. Since $x$ was arbitrary, no repelling
  fixed point of $X$ admits a critically fixed direction. From
  \cite[Thm.~A]{Faber_Patra_fixed_structure}, $X$ contains at least~3 classical
  fixed points, which is absurd. This final contradiction establishes the
  lemma.
\end{proof}
    
We use the term \textbf{metric graph} (also called a \textbf{metrized graph})
in the sense of \cite[Ch.~3]{Baker-Rumely_BerkBook_2010}. In particular, such a
graph has finite length. A \textbf{metric tree} is a metric graph whose
underlying topological space is contractible. A \textbf{leaf} of a metric graph
is a point with a unique tangent direction.
    
\begin{lemma}
  \label{lem:metric_fixed_pt}
  Let $\Gamma$ be a metric tree with positive length and leaves $x_1, \ldots,
  x_n$.  Suppose $\varphi : \Gamma \to \Gamma$ is a self-map of $\Gamma$ with
  the following properties:
  \begin{itemize}
      \item $\varphi$ is continuous and piecewise affine;
      \item $\varphi(x_i) = x_i$ for $i = 1, \ldots, n$; and 
      \item the derivative of $\varphi$ at $x_i$ (in the unique direction
        pointing into $\Gamma$) is strictly greater than~1 for $i = 1, \ldots,
        n$.
  \end{itemize}
  Then $\varphi$ admits a fixed point in $\Gamma \smallsetminus \{x_1, \ldots, x_n\}$.
\end{lemma}
    
\begin{proof}
  We proceed by induction on $n$. Suppose first that $n = 2$. Then $\Gamma =
  [x_1, x_2]$ is a segment, and there is an isometry $\Gamma \simeq [0,\ell]$
  for some $\ell > 0$. Under this isometry, our self-map of $\Gamma$
  corresponds to a continuous piecewise affine function $g : [0,\ell] \to
  [0,\ell]$ with
  \[
  g(0) = 0, \quad g(\ell) = \ell, \quad g'(0) > 1, \quad g'(\ell) > 1.
  \]
  Here we have abused notation by writing $g'$ for the 1-sided derivatives at
  $0$ and $\ell$, as there is no ambiguity. The Intermediate Value Theorem
  shows that $g(x) - x$ has a zero in $(0,\ell)$, and hence $g$ has a fixed
  point in this interval.

  Suppose now that the result holds for every such tree with at most $n \ge 2$
  leaves, and let us assume that $\Gamma$ has $n+1$ leaves. In particular, this
  means $\Gamma$ has at least one point $y$ with valence strictly larger
  than~2. We may decompose $\Gamma$ as
  \[
     \Gamma = \Gamma_1 \cup \Gamma_2 \cup \Gamma_3, 
  \]
  where each $\Gamma_i$ is a metric tree with positive length and $\Gamma_i
  \cap \Gamma_j = \{y\}$ for $i \ne j$. Then any leaf of $\Gamma_{i}$ is either
  $y$ or a leaf of $\Gamma$. For convenience, define $\Gamma_{ij} = \Gamma_i
  \cup \Gamma_j$. Observe that for $i \neq j$, any leaf of $\Gamma_{ij}$ must
  be a leaf of $\Gamma$. As a consequence, each $\Gamma_{ij}$ has at most $n$
  leaves.
        
  Define $r_{12} : \Gamma \to \Gamma_{12}$ to be the canonical retraction from
  $\Gamma$ to $\Gamma_{12}$. Concretely, $r_{12}(x) = x$ for $x \in
  \Gamma_{12}$ and $r_{12}(x) = y$ for $x \in \Gamma_3$. Write $\varphi_{12} :
  \Gamma_{12} \to \Gamma$ for the restriction of $\varphi$ to
  $\Gamma_{12}$. The composition $r_{12} \circ \varphi_{12}$ is a continuous
  and piecewise affine self-map of $\Gamma_{12}$, and it agrees with $\varphi$
  in a neighborhood of each leaf of $\Gamma_{12}$. In particular, it satisfies
  all of the hypotheses of the proposition. By our induction hypothesis, the
  map $r_{12} \circ \varphi_{12}$ has a fixed point $x$ that is not a leaf of
  $\Gamma_{12}$. If $x \in \Gamma_{12} \smallsetminus \{y\}$, then $\varphi(x)
  = x$, and we are finished. Otherwise, $y$ is fixed by $r_{12} \circ
  \varphi_{12}$. By the definition of the retraction, we conclude that
  $\varphi(y) \in \Gamma_3$. Again, if $\varphi(y) = y$, we are done. So we may
  assume that $\varphi(y) \in \Gamma_3 \smallsetminus \{y\}$.
        
  Now we make the same argument again, replacing $\Gamma_{12}$ with
  $\Gamma_{23}$. Then we either find a fixed point of $\varphi$ in
  $\Gamma_{23}$, or else we see that $\varphi(y) \in \Gamma_1 \smallsetminus
  \{y\}$. The latter cannot hold since $\varphi(y) \in \Gamma_3 \smallsetminus
  \{y\}$. The proof is complete.
\end{proof} 
    
\begin{proof}[Proof of Theorem~C]
  If $U$ is not a minimal compressive domain, then replace it with a
  compressive domain $V$ with $V \subsetneq U$. Repeat as necessary until $U$
  is minimal. Note that this process terminates because there are only finitely
  many repelling type~II fixed points for $f$
  \cite[Cor.~6.2]{Rumely_new_equivariant}.
       
  If $U = \BerkK$, the result is evident. If instead $U$ is a disk, the result
  is immediate from Rumely's First Identification Lemma. Suppose for the
  remainder of the proof that $U$ has at least two boundary points. Write
  $\Gamma$ for the connected hull of $x_1, \ldots, x_n$, i.e., the smallest
  connected subset of $\BerkK$ containing $x_{1}, \ldots, x_{n}$.  Write $r :
  \BerkK \to \Gamma$ for the retraction to $\Gamma$. Then $r \circ f$ is
  continuous and piecewise affine for the path distance metric on $\Gamma$
  \cite[Cor.~9.21]{Baker-Rumely_BerkBook_2010}. The fact that $U$ is a 
  compressive domain shows that $r \circ f$ satisfies the hypotheses of
  Lemma~\ref{lem:metric_fixed_pt}. It follows that there is a point $x \in
  \Gamma \smallsetminus \{x_1, \ldots, x_n\}$ such that $r \circ f(x) = x$. If
  $f(x) = x$, then Lemma~\ref{lem:fixed_implies_classical} shows there is a
  classical fixed point in $U$. If instead $f(x) \ne x$, then there is a
  direction $\vv \in T_x$ such that $B_x(\vv)^- \subset U$ and $f(x) \in
  B_x(\vv)^-$. Rumely's Second Identification Lemma
  \cite[Lem.~2.2]{Rumely_new_equivariant} shows that $B_x(\vv)^-$ contains a
  classical fixed point of $f$.
\end{proof}

Combining Theorem~C with Proposition~\ref{prop:compressive_count} immediately
yields the following estimate for the sum of the surplus multiplicities of the
inward directions of a compressive domain.

\begin{corollary}
  Let $f$ be a rational function of degree $d \ge 1$, and let $U =
  \bigcap_{i=1}^n B_{x_i}(\vv_i)^-$ be a compressive domain for $f$. Write
  $s_f(\vv_i)$ for the surplus multiplicity of the direction $\vv_i$. Then
  \[
     \sum_{i=1}^n s_f(\vv_i) \ge  d(n-1).
   \]
\end{corollary}


\section{Bounding the number of hyperbolic components}
\label{sec:hyperbolic}

Recall that a hyperbolic component can only be present in $\Fix(f)$ if $p =
\reschar(K)$ is positive \cite[Thm.~B]{Faber_Patra_fixed_structure}. The
presence of hyperbolic components forces a kind of separation of the classical
fixed points. We give two preliminary results in this vein and then deduce
Theorem~D.

\begin{lemma}
  \label{lem:critically_fixed_bound}
  Suppose that $X$ is a hyperbolic component of $\Fix(f)$. Then $\Ncf(f,X) \ge
  p+1$.
\end{lemma}

\begin{proof}
  Let $x_1, \ldots, x_m$ be the type~II repelling fixed points in $X$. Then  
  \[
     \sum_{i=1}^m \deg_f(x_i) = m - 1 + pj
  \]
  for some integer $j$ \cite[Prop.~7.4]{Faber_Patra_fixed_structure}. Since
  $\deg_f(x_i) \ge 2$ for each $i$, we find that
  \[
    pj = \sum_{i=1}^m \deg_f(x_i) - m + 1 \ge m + 1 \ge 2.
  \]
  In particular, $j \ge 1$.

  Viewing $X$ as a metric tree with vertex set $\{x_1, \ldots, x_m\}$
  \cite[Prop.~7.1]{Faber_Patra_fixed_structure}, we see that
  \begin{align*}
    \Ncf(f,X) &\ge \sum_{i=1}^m \Ncf(f,x_i) \\
    &= \sum_{i=1}^m \left( \deg_f(x_i) + 1 - v_X(x_i)\right) \\
    &= 2m - 1 + pj - \sum_i v_X(x_i),
  \end{align*}
  where $v_X(x_i)$ denotes the valence of $x_i$ in $X$.
  Since $X$ has $m$ vertices and $m-1$ edges, the handshaking lemma from graph
  theory gives
  \[
    \Ncf(f,X) \ge pj + 1. \qedhere
  \]
\end{proof}


      

\begin{proposition}
  \label{prop:compressive_classical_bound}
  Let $U$ be a compressive domain for $f$. Write $s$ for the number of
  hyperbolic components of $\Fix(f)$ in $U$. Then the number of distinct
  classical fixed points in $U$ is at least $ps + 1$.
\end{proposition}

\begin{proof}
  If $s = 0$, then Theorem~C shows that $U$ contains at least one classical
  fixed point and the proof is complete. For the remainder of the argument, we
  assume that $s \ge 1$.

  Let $\Gamma$ be the connected hull of the boundary points of $U$ and the
  classical fixed points in $U$. Write $X_1, \ldots, X_s$ for the hyperbolic
  components of $f$ in $U$.  Each type~II repelling fixed point $x \in X_i$ has
  at least three simply fixed directions, since there are no id-indifferent
  fixed points in $X_{i}$ \cite[Prop.~7.1]{Faber_Patra_fixed_structure}. Thus,
  Rumely's First Identification Lemma shows that there are classical fixed
  points in at least three distinct directions at $x$. That is, $x$ is a branch
  point of $\Gamma$. It follows that $X_i \subset \Gamma$ since $X_i$ is the
  connected hull of its type~II repelling fixed points
  \cite[Prop.~7.1]{Faber_Patra_fixed_structure}.

  Write
  \[
  \Gamma' = \Gamma \smallsetminus \bigcup_{i=1}^s  X_i.
  \]
  We are going to give upper and lower bounds for the number of connected
  components of $\Gamma'$, which we denote by $h$.

  Write $n_i$ for the number of connected components of $\Gamma \smallsetminus
  X_i$. Each such component corresponds to a point $x \in X_i$ and a direction
  $\vv \in T_x$ that contains a classical fixed point. Since the direction
  points out of $X_i$, it is either critically fixed or else a shearing
  direction. Conversely, every critically fixed direction or shearing direction
  at a point of $X_i$ corresponds to a connected component of $\Gamma
  \smallsetminus X_i$.  That is,
  \[
     n_i  = \Ncf(f,X_i) + \Nshear(f,X_i),
  \]
  where $\Nshear(f,X_i)$ is the total number of shearing directions at points
  of $X_i$.  Since $\Gamma$ is a topological tree and the $X_i$ are pairwise
  disjoint, an inductive argument shows that the number of connected components
  of $\Gamma'$ is $1 + \sum(n_i - 1)$. We have
  \begin{align}
    \label{eq:components_lower_bound}
    h &= 1 + \sum_{i=1}^s (n_i - 1) \notag \\ 
    &= 1 - s + \sum_{i=1}^s \Ncf(f,X_i) + \sum_{i=1}^s \Nshear(f,X_i)  \notag \\
    &\ge 1 - s + s(p+1) + \sum_{i=1}^s \Nshear(f,X_i) \notag \\
    &= 1 + ps + \sum_{i=1}^s \Nshear(f,X_i),
  \end{align}
  where the inequality follows from Lemma~\ref{lem:critically_fixed_bound}.
  
  Turning to the upper bound, we write $c$ for the number of distinct classical
  fixed points of $f$ contained in $U$. There are at most $c$ components of
  $\Gamma'$ that contain a classical fixed point. Let $V$ be a component that
  does not contain a classical fixed point. We claim that the inward direction
  at some boundary point of $V$ is shearing.  Let $x$ be a boundary point of
  $V$ and $\vv \in T_x$ the direction that contains $V$. Then $x$ is a type~II
  fixed point for $f$. If $x$ is a boundary point of $U$, then $\vv$ is
  critically fixed. If instead $x$ lies in some hyperbolic component, then
  there is a classical fixed point in the direction $\vv$, so that $\vv$ is
  either critically fixed or shearing. But if all inward directions at all
  boundary points of $V$ are critically fixed, then Theorem~C shows that $V$
  contains a classical fixed point, which is absurd. So some inward direction
  at a boundary point of $V$ is shearing, and that boundary point lies on a
  hyperbolic component. It follows that
  \begin{equation}
    \label{eq:components_upper_bound}
    h \le c + \sum_{i=1}^s \Nshear(f,X_i).
  \end{equation}

  Comparing \eqref{eq:components_lower_bound} and
  \eqref{eq:components_upper_bound}, we immediately obtain $c \ge 1 + ps$, as
  desired.
\end{proof}

\begin{proof}[Proof of Theorem~D]
  Write $s$ for the number of hyperbolic components in $\Fix(f)$, and write $c$
  for the number of distinct classical fixed points of $f$. Applying
  Proposition~\ref{prop:compressive_classical_bound} with $U = \BerkK$ and
  using the fact that a rational function has at most $d+1$ distinct classical
  fixed points, we see that
  \[
    ps+1 \le c \le d + 1.
    \]
  Thus, $s \le d / p$. Since every non-hyperbolic component contains at least
  one of the classical fixed points, we conclude that $\Fix(f)$ has at most $c
  + d/p \le d + 1 + d/p$ connected components.
\end{proof}
  

\section{Polynomials}
\label{sec:polynomials}

Juan Rivera-Letelier suggested that a mass formula for repelling fixed points
of polynomials gives a different approach to bounding the number of components
of the fixed locus. This mass formula does not appear elsewhere in the
literature. 
    
\begin{proposition}[Rivera-Letelier]
  \label{prop:RL_repelling}
  Let $f \in K[z]$ be a polynomial of degree $d \ge 2$. Then
  \[
    \sum_{\substack{x \in \Fix(f) \\ x \text{ repelling}}} \deg_f(x) = d.
  \]
\end{proposition}
    
\begin{proof}
  Recall that we can impose a partial order on points of $\BerkK$: write $x
  \preceq y$ if $\|g\|_x \le \|g\|_y$ for any polynomial $g \in K[z]$, where
  $\|\cdot\|_x$ is the seminorm associated with $x$, and similarly for $y$. Points
  of type~I, II, and III correspond to supremum seminorms on disks, which means
  $\zeta_{a,r} \preceq \zeta_{b,s}$ if and only if the corresponding classical
  disks satisfy $D(a,r) \subset D(b,s)$. Since $f$ has a totally ramified fixed
  point at infinity, the image of a disk $D(a,r)$ is always a disk, and
  inclusion of disks is preserved by $f$. In terms of the partial ordering,
  this means $f(x) \preceq f(y)$ whenever $x \preceq y$. Moreover, $f$ is
  injective on the arc $[x,y]$ if $x \preceq y$.

  Let $x$ be a classical repelling fixed point. We claim that $x$ and $\infty$
  are the only fixed points on the arc $[x,\infty]$. Indeed, if $\lambda$ is
  the fixed-point multiplier at $x$, then the proof of
  \cite[Lem.~2.4]{Faber_Patra_fixed_structure} shows that
  $f(\zeta_{x,\varepsilon}) = \zeta_{x,|\lambda|\varepsilon}$ for all
  $\varepsilon$ sufficiently small. So if $\zeta_{x,r}$ were a fixed point in
  $(x,\infty)$, then we would have the following inequality using the path
  distance metric \cite[\S2.7]{Baker-Rumely_BerkBook_2010}:
  \[
    \rho(\zeta_{x,\varepsilon}, \zeta_{x,r}) = \log \frac{r}{\varepsilon} 
    > \log \frac{r}{|\lambda|\varepsilon}
    = \rho(\zeta_{x,|\lambda|\varepsilon}, \zeta_{x,r}) =
    \rho(f(\zeta_{x,\varepsilon}), f(\zeta_{x,r})) \ge \rho(\zeta_{x,\varepsilon}, \zeta_{x,r}).
  \]
  The final inequality holds because the action of $f$ on the arc
  $[\zeta_{x,\varepsilon}, \zeta_{x,r}]$ is injective and locally non-contracting 
  for the path distance metric. Evidently we have encountered a contradiction.

  Next, let $x$ be an attracting or indifferent classical fixed point that is
  distinct from $\infty$. We claim that there is a single repelling type~II
  fixed point on the arc $[x,\infty]$.  Let $U \subset [x,\infty]$ be the set
  of points that are attracted to $\infty$ under iteration of $f$. As $f$ is
  order preserving, $U$ is a connected neighborhood of $\infty$, and $U$ is
  disjoint from a neighborhood of $x$
  \cite[Lem.~2.4]{Faber_Patra_fixed_structure}. It follows that $y := \inf U
  \ne x$. We claim that $y$ is fixed by $f$. On one hand, $f(y) \not\in U$, so
  $f(y) \preceq y$. On the other hand, $f$ is continuous and order-preserving
  on $[y, \infty]$, so if $f(y) \prec y$, then there is $z \in (y, \infty)$
  such that $f(y) \prec f(z) \prec y$. But this contradicts the fact that $z
  \in U$, and hence must be attracted to $\infty$. Thus $f(y) = y$. The
  direction toward $\infty$ is fixed under the action of $T_y f$. If $y$ were
  an indifferent fixed point, then every point in $(y,\infty)$ sufficiently
  close to $y$ would also be fixed, which contradicts the definition of $y$. We
  conclude that $y$ is a type~II repelling fixed point. Suppose there is a
  second type~II repelling fixed point $y' \in (x,\infty)$. Length scales by
  the local degree along segments, so all points of the interval $(y',y)$ are
  fixed and indifferent. This means $\deg_f(y') > 1$ and $\deg_f(y) > 1$, while
  $\deg_f(w) = 1$ for $w \in (y',y)$. But this is impossible since the local
  degree is non-increasing on the segment $(\infty,x)$
  \cite[Thm.~9.42]{Baker-Rumely_BerkBook_2010}.
  
  Suppose that $y$ is a type~II repelling fixed point for $f$, and let $\vv$ be
  a direction at $y$ that does not contain infinity. Since $\infty$ is totally
  ramified, there is no pole in the direction $\vv$, and we have $s_f(\vv) =
  0$. By Rumely's First Identification Lemma, we see that
  \[
     F_f(\vv) = \tilde F_f(\vv).
  \]
  Summing over all directions not containing $\infty$ and using the fact that
  $\infty$ is a simple fixed point for both $f$ and $T_y f$, we see that
  \[
  \#\{\text{classical fixed points } x \preceq y\} =
  \#\{\text{finite fixed points of } T_y f\} = \deg_f(y).
  \]

  We have proved that
  \begin{align*}
  \sum_{\substack{x \in \Fix(f) \\ \text{ repelling}}} \deg_f(x)
  &= \sum_{\substack{x \in \Fix(f) \\ \text{ classical repelling}}} 1
  + \sum_{\substack{y \in \Fix(f) \\ \text{ repelling type~II}}} \deg_f(y) \\
  &= \sum_{\substack{x \in \Fix(f) \\ \text{ classical repelling}}} 1
  + \sum_{\substack{y \in \Fix(f) \\  \text{ repelling type~II}}}
  \#\{\text{classical fixed points } x \preceq y\} = d.
  \end{align*}
  The final equality follows from the fact that the second sum counts all finite
  attracting or indifferent classical fixed points.
\end{proof}

\begin{proposition}
  \label{prop:poly_hyperbolic}
  Let $f \in K[z]$ be a non-constant polynomial, and suppose there is a
  hyperbolic component $X$ in $\Fix(f)$. Then $X = \{x\}$ for some type~II
  repelling fixed point $x$. Moreover, $x$ admits no shearing direction and
  $T_x f$ is purely inseparable.
\end{proposition}

\begin{proof}
  Let $x \in X$ be a type~II repelling fixed point. We claim that $x$ admits no
  shearing direction. Without loss of generality, we may change coordinate via
  an affine invertible map, so that $x = \zeta_{0,1}$. The reduction of $f$ is
  a polynomial of degree at least~2, which shows that the direction toward
  infinity is critically fixed. Since the point at infinity is fixed and
  totally ramified, no other direction at $x$ can be a bad direction. By
  Rumely's First Identification Lemma, a direction at $x$ contains a classical
  fixed point if and only if it is a fixed direction. In particular, $x$ admits
  no shearing direction. 

  If $X \ne \{x\}$, then $X$ contains a second type~II repelling fixed point
  $y$ \cite[Prop.~7.1]{Faber_Patra_fixed_structure}. Consider the nontrivial
  arc $[x,y] \subset X$. Let $\vv \in T_x$ contain $y$. Then $B_x(\vv)^-$ does
  not contain $\infty$ since the direction towards $\infty$ is critically
  fixed. Now $y$ is a branch point of $\Gamma_{\Fix}$ (as is any repelling
  type~II fixed point of a hyperbolic component), which means $B_x(\vv)^-$
  contains at least two classical fixed points. We have already established
  that $\vv$ is not a bad direction, so Rumely's First Identification Lemma
  implies that the fixed-point multiplicity of $\vv$ for $T_x f$ is greater
  than~1. Then our Second Indifference Lemma
  \cite[Lem.~3.2]{Faber_Patra_fixed_structure} implies that the arc $(x,y)$
  contains id-indifferent fixed points, which contradicts the fact that $X$ is
  a hyperbolic component \cite[Prop.~7.1]{Faber_Patra_fixed_structure}.  Thus,
  $X = \{x\}$.

  Let $g = \tilde f$ be the reduction of $f$, so that $T_x f(z) = g(z)$.  Since
  $X = \{x\}$, every fixed direction at $x$ is critically fixed. This implies
  the roots of $g(z) - z$ are simple, and that $g(z) - z$ divides $g'(z)$. If
  $g'(z)$ is nonzero, then its degree is strictly smaller than that of $g(z) -
  z$, a contradiction. Thus $g'(z) = 0$ and $T_x f$ is purely inseparable.
\end{proof}

\begin{corollary}
  \label{cor:juan_bound}
  Let $f \in K[z]$ be a polynomial of degree $d \ge 2$. Then the number of
  components of the fixed locus of $f$ is at most
  \[
  \begin{cases}
    d + 1 & \text{ if $p = 0$ or $p > d$} \\
    d + \lfloor \frac{d}{p}\rfloor + 1 & \text{ if $0 < p \le d$}.
  \end{cases}
  \]
\end{corollary}

\begin{proof}
  If $p = 0$ or $p > d$, then $\Fix(f)$ has no hyperbolic component by
  Theorem~B of \cite{Faber_Patra_fixed_structure}. Assume that $0 < p \le d$.
  Suppose that $\Fix(f)$ has $r$ hyperbolic components, which are of the form
  $\{x_1\}, \ldots, \{x_r\}$ with $p \mid \deg_f(x_i)$ for each $i$ by
  Proposition~\ref{prop:poly_hyperbolic}. The mass formula in
  Proposition~\ref{prop:RL_repelling} implies that
  \[
     d = \sum_{\substack{x \in \Fix(f) \\ x \text{ repelling}}} \deg_f(x) \ge pr.
  \]
  That is, $r \le d / p$. There are at most $d+1$ non-hyperbolic components
  since each contains a classical fixed point, so we have at most $d + d/p + 1$
  components in total.
\end{proof}


\section{Extremal Examples}
\label{sec:examples}

In this section, we give examples of polynomial maps of every degree $d \ge 1$
that achieve the bounds in Theorem~B \cite{Faber_Patra_fixed_structure},
Theorem~D, and Corollary~\ref{cor:juan_bound}. We take $K$ to be an arbitrary
algebraically closed field that is complete with respect to a nontrivial
nonarchimedean absolute value, and we write $p$ for the residue characteristic
of $K$. Our examples will naturally split into two cases: large residue
characteristic ($p > d$ or $p = 0$) and small residue characteristic ($0 < p
\le d$).


\subsection{Large residue characteristic}
Choose elements $w_1, \ldots, w_d \in K^\times$ that satisfy
\[
   0 < |w_d| < \cdots < |w_1|, 
\]
and let $c \in K^\times$ satisfy $|c| > |w_d|^{-1}$.
Consider the polynomial
\[
  f(z) = z + c \prod_{i=1}^{d} (w_i z - 1).
\]
The classical fixed points of $f$ are $\infty, w_1^{-1}, \ldots, w_d^{-1}$.

The derivative at a fixed point $w_j^{-1}$ is
\[
  f'(w_j^{-1}) = 1 + c w_j \prod_{i \neq j} (w_i w_j^{-1} - 1).
\]
Our assumptions on the elements $w_1, \ldots, w_j$ and $c$ show that $|c w_j| >
1$ and each factor $|w_i w_j^{-1} - 1| = \max\{1, |w_i w_j^{-1}|\} \ge 1$. Thus,
\[
|f'(w_j^{-1})| = |c w_j| \cdot \prod_{i < j} |w_i w_j^{-1}|
\ge |c w_j| > 1.
\]
Then $w_1^{-1}, \ldots, w_d^{-1}$ are repelling fixed points, while $\infty$ is
attracting. It follows that each of these points is isolated in the fixed locus
\cite[Lem.~2.4]{Faber_Patra_fixed_structure}, so that $\Fix(f)$ has at least
$d+1$ components.

If we assume that $p = 0$ or $p > d$, then we have produced a polynomial that
achieves the bound in Theorem~B of \cite{Faber_Patra_fixed_structure} as well
as the bound in Corollary~\ref{cor:juan_bound}.


\subsection{Small residue characteristic}

It will be convenient in the following calculations to write $o(1)$ for
an unspecified polynomial with coefficients in $\mm$. Such a polynomial will
vanish when we reduce coefficients.

Assume that $0 < p \le d$. Write $d = mp + r$ for some nonnegative integer $r <
p$. Choose elements $x_1, \ldots, x_m$ and $w_1, \ldots, w_r$ in $K$ such that
\[
   0 < |w_r| < \cdots < |w_1| < |x_m| < \cdots < |x_1| < 1,
\]
and define
\[
   c = \frac{(-1)^{(m-1)p}}{x_1^p}.
\]
Finally, define $y_1, \ldots, y_m \in K^\times$ by

\[
    y_{1}^{p-1} = 1
\]

\[
   y_j^{p-1} = \frac{x_j^{(j-2)p}(-1)^{(m-j)p}}{c(x_1 \cdots x_{j-1})^p};
\]
here we choose any $(p-1)$-st root of the expression on the right to represent
$y_j$. Then $|c| > 1$, $|y_1| = 1$, and for $j \ge 2$, we have
\[
|y_j|^{p-1} = |x_j|^{(j-2)p}|x_2 \cdots x_{j-1}|^{-p} = \prod_{i=2}^{j-1} |x_j / x_i|^p \le 1.
\]
In particular, $|x_j y_j| < 1$ for all $j$.

Define
\[
   f(z) = z + f_0(z)f_1(z), 
\]
where
\begin{align*}
  f_0(z) &= c \prod_{i=1}^m \left[ (x_i z - 1)^p - (x_iy_i)^{p-1}(x_iz-1)\right] \\
  f_1(z) &= (-1)^r \prod_{i=1}^r (w_i z - 1).
\end{align*}
We claim that $\Fix(f)$ consists of $d+1$ classical components and $m = \lfloor
d/p\rfloor$ hyperbolic components. The hyperbolic components are
$\{\zeta_{x_i^{-1},|y_i|}\}$ for $i = 1, \ldots, m$.

The roots of $f_0, f_1$ and the point at $\infty$ are the classical fixed
points of $f$. The roots of $f_1(z)$ are $w_1^{-1}, \ldots, w_r^{-1}$. Fix an
index $j \le r$. We see that
\begin{align*}
  f'(w_j^{-1}) &= 1 + (-1)^r f_0 (w_j^{-1}) \prod_{i \neq j} (w_i w_j^{-1} - 1) \\ &=
  1 + (-1)^r f_0 (w_j^{-1}) \left( (-1)^{r-j} + o(1)\right)
  \prod_{i < j} (w_i w_j^{-1} - 1),
\end{align*}
where $|w_i w_j^{-1} - 1| = |w_i w_j^{-1}| > 1$.  Since $|w_j| < |x_i|$ for all
$i$, we have
\[
|f_0 (w_j^{-1})| = |c| \prod_{i=1}^m \left| (x_i w_j^{-1} - 1)^p - (x_iy_i)^{p-1}(x_iw_j^{-1} - 1)\right|
= |c| \prod_{i=1}^m |x_i w_j^{-1}|^p > 1.
\]
It follows that
\[
   |f'(w_j^{-1})| = |c| \prod_{i=1}^m |x_i w_j^{-1}|^p \prod_{i < j} |w_i w_j^{-1}| > 1,
\]
which means the fixed point $w_j^{-1}$ is repelling and isolated in $\Fix(f)$.

Next we argue that for each $j = 1, \ldots, m$, the point
$\zeta_{x_j^{-1},|y_j|}$ is an inseparable fixed point for $f$ of local
degree~$p$. To that end, we fix an index $1 \le j \le m$ and define
\[
  g_j(z) := y_j^{-1} f(y_j z + x_j^{-1}) - x_j^{-1}y_j^{-1} = z + y_j^{-1}f_0(y_j z + x_j^{-1}) f_1(y_j z + x_j^{-1}). 
\]  
It suffices to show that $g_j$ has integral coefficients and reduction $z^p$.
Since the absolute values of the $x_i$'s are strictly decreasing as the index
increases, the factor indexed by $i$ in the product defining $f_0(y_j z
+ x_j^{-1})$ satisfies
\begin{align*}
  (i=j) : \quad & (x_iy_j z)^p - (x_j y_j)^{p-1}(x_j y_jz) = (x_j y_j)^p (z^p - z) \\
  (i < j) : \quad  & (x_i x_j^{-1})^p\left[ (x_j y_j z + 1 - x_j x_i^{-1})^p - (x_j y_i)^{p-1}(x_j y_j z + 1 - x_j x_i^{-1})\right]
  = (x_i x_j^{-1})^p(1+o(1)) \\
  (i > j) : \quad & (-1)^p + o(1). 
\end{align*}
Putting these three quantities together and using the given expressions for $c$
and $y_j$, we find that
\begin{align*}
y_j^{-1}f_0(y_j z + x_j^{-1}) &= y_j^{-1} c (x_1 \cdots x_{j-1})^p
x_j^{-(j-1)p}(x_j y_j)^p (z^p - z)\left( (-1)^{(m-j)p} + o(1)\right) \\
&= z^p - z + o(1).
\end{align*}
Since $|w_i x_j^{-1}| < 1$ for all $i,j$ and $|y_j| < |x_j^{-1}|$, the
definition of $f_1$ gives
\[
f_1(y_j z + x_j^{-1}) = (-1)^r \prod_{i=1}^r (w_i y_j z + w_i x_j^{-1} - 1) = 1 + o(1).
\]
Combining all of this analysis shows that $g_j(z) = z^p + o(1)$, as desired.

We are now able to conclude that each point $\zeta_j := \zeta_{x_j^{-1},|y_j|}$
forms a hyperbolic component of $\Fix(f)$. Moreover, since $|y_j| < |x_j^{-1}|$
for all $j$, it follows that the point $\zeta_j$ does not lie on the arc $[0,
  \infty]$. The fact that $f$ has inseparable reduction of local degree~$p$ at
$\zeta_j$ implies that there are $p$ critically fixed directions in
$T_{\zeta_j}$ that do not contain infinity. Moreover, these are all good
directions since $f$ is a polynomial. Rumely's First Identification Lemma
shows that each such direction contains a single classical fixed point, which
must be attracting. In summary, we have now exhibited $mp + r + 1= d+1$
attracting or repelling fixed points (including the point at infinity), and we
have found $m = \lfloor d/p \rfloor$ hyperbolic components of the fixed
locus. Thus, $f$ has at least $d + \lfloor d/p \rfloor + 1$ components in its
fixed locus, which agrees with the bound in Theorem~D. 


\section{Crucial points and compressive domains}
\label{sec:crucial_pts}

The purpose of this section is to prove Theorem~E. We begin with a brief study
of the boundary of the id-indifference locus for a rational function $f$ in
\S\ref{subsec:boundary}, and we follow in \S\ref{subsec:compressive_trees} with
an analysis of a particularly important tree inside a compressive domain for
$f$. We conclude with the proof of Theorem~E in
\S\ref{subsec:proof_of_Theorem_E}.


\subsection{The boundary of the id-indifference locus}
\label{subsec:boundary}

Recall that a fixed point $x$ for $f$ is called \textbf{id-indifferent}
if the associated map on tangent directions $T_x f$ is the identity. The
\textbf{id-indifference locus} is the set $I_f$ of all id-indifferent fixed
points for $f$.

The boundary of $I_f$ consists entirely of fixed points.  By the First
Indifference Lemma, a classical fixed point lies in the boundary of $I_f$ if
and only if its multiplier is congruent to $1$ modulo the maximal ideal. By the
Second Indifference Lemma, a type~II point $x$ lies in the boundary of $I_f$ if
and only if it is not id-indifferent and there is a direction $\vv \in T_x$
such that the fixed point multiplicity for $\vv$ under $T_x f$ is greater
than~1.

The presence of id-indifferent fixed points in $U$ introduces substantial
complications into the proof of Theorem~E. We prove two preliminary results to
address some of these complications.

\begin{proposition}
  \label{prop:num_classically_indifferent}
  Let $J$ be a connected component of $I_f$. Write $u_1, \ldots, u_m$ for the
  boundary points of $J$ that are fixed and repelling under $f$, and write
  $\vv_i$ for the direction at $u_i$ that meets $J$. The number of
  classical fixed points in the boundary of $J$, counted according to
  fixed-point multiplicity, is given by
  \[
     2 - m + \sum_{i=1}^m \Big(\tilde F_f(\vv_i) - 1\Big).
  \]
\end{proposition}

\begin{proof}
  Without further comment in this proof, classical fixed points will be counted
  according to fixed-point multiplicity.

  Note that $J$ is entirely contained inside a single component $X$ of
  $\Fix(f)$. Write $\bar J$ for the closure of $J$. We will count the number of
  classical fixed points in $\bar J$ by first counting the number in $X$
  that lie outside $\bar J$. Given a type~II point $x \in X$ and a direction $\vv
  \in T_x$, we write $F_f(\vv,X)$ for the number of classical fixed points in
  $X \cap B_x(\vv)^-$, counted according to fixed-point multiplicity. Lemma~6.1
  of \cite{Faber_Patra_fixed_structure} shows that if $x$ is not id-indifferent
  and $\vv$ is not critically fixed, then
  \begin{equation}
    \label{eq:X_classical_fixed}
  F_f(\vv,X) = \tilde F_f(\vv) +
  \sum_{\substack{y \in X \cap B_x(\vv)^- \\ \text{type~II repelling}}}
  \left( \deg_f(y) - 1 - \Ncf(f,y)\right).
  \end{equation}
  Note that $F_f(\vv,X) = 0$ if $\vv$ is critically fixed, since $X \cap
  B_x(\vv)^- = \varnothing$ in that case. Similarly, $F_f(\vv,X) = 0$ if $x$ is
  additively indifferent and $\vv$ is not equal to the unique direction that
  meets $J$, since then $\vv$ is not fixed. Applying
  \eqref{eq:X_classical_fixed} at each non-classical boundary point $u_i$ of $J$ and each
  non-critically fixed direction $\vv \ne \vv_i$, we find that the number of
  classical fixed points in $X \smallsetminus \bar J$ is
  \begin{align}
    \label{eq:outside_I}
     \sum_{i=1}^m & \left[\sum_{\substack{\vv \in T_{u_i} \smallsetminus \vv_i \\ \text{non-critically fixed}}} \left( \tilde F_f(\vv) +
  \sum_{\substack{y \in X \cap B_{u_{i}}(\vv)^- \\ \text{type~II repelling}}}
  \left( \deg_f(y) - 1 - \Ncf(f,y)\right) \right)\right] \\
     &= \sum_{i=1}^m \left(\deg_f(u_i) + 1 - \tilde F_f(\vv_i) - \Ncf(f,u_i)\right)
     + \sum_{\substack{y \in X \smallsetminus \bar J \\ \text{type~II repelling}}}
     \left( \deg_f(y) - 1 - \Ncf(f,y)\right). \notag
  \end{align}
  
  Theorem~A of \cite{Faber_Patra_fixed_structure} shows that the total number
  of classical fixed points in $X$, counted according to fixed-point
  multiplicity, is
  \begin{equation}
    \label{eq:thmA}
     2 + \sum_{\substack{y \in X  \\ \text{type~II repelling}}}
     \left( \deg_f(y) - 1 - \Ncf(f,y)\right).
  \end{equation}
  Subtracting \eqref{eq:outside_I} from \eqref{eq:thmA} gives the number of
  classical fixed points in $\bar J$:
  \[
     2 + \sum_{\substack{y \in  \bar J  \\ \text{type~II repelling}}}
     \left( \deg_f(y) - 1 - \Ncf(f,y)\right)
     - \sum_{i=1}^m \left(\deg_f(u_i) + 1 - \tilde F_f(\vv_i) - \Ncf(f,u_i)\right). 
  \]
  Note that only the boundary points of $J$ can be type~II repelling. Hence,
  this expression equals
  \[
     2 + \sum_{i=1}^m \left( \tilde F_f(\vv_i) - 2\right),
  \]
  which is equivalent to the expression in the statement of the proposition.
\end{proof}

\begin{corollary}
  \label{cor:bound_on_no_of_id_indiff_comp}
  Each component of $I_f$ has at least two classical fixed points in its
  boundary, counted according to fixed-point multiplicity. In particular, the
  number of components of $I_f$ is bounded above by $(d+1)/2$, where $d =
  \deg(f)$.
\end{corollary}

\begin{proof}
  The first claim is immediate from
  Proposition~\ref{prop:num_classically_indifferent}. The second claim follows
  from the fact that $f$ has exactly $d+1$ classical fixed points, counted
  according to fixed-point multiplicity.
\end{proof}

\begin{proposition}
  \label{prop:crucial_wt_formula}
  Let $f$ be a non-constant rational function. Let $x \in \BerkK$ be a type~II
  fixed point for $f$ that is not id-indifferent. If there are $r$ directions
  at $x$ containing a classical fixed point, then the crucial weight of $x$ is
  given by
  \[
     \wR(x) = r - 2 + \sum_{\substack{\vv \in T_x \\ T_x f(\vv) = \vv}} \Big(\tilde F_f(\vv) - 1\Big),
  \]
  where $\tilde F_f(\vv)$ is the multiplicity of $\vv$ as a fixed point of
  $T_x f$.
\end{proposition}
    
\begin{proof}
  Any direction at $x$ containing a classical fixed point is either a fixed
  direction or a shearing direction. Let $\vv_1, \ldots, \vv_k$ be the distinct
  fixed directions at $x$; the remaining $r-k$ are shearing
  directions. Counting the fixed points of $T_x f$ with multiplicity yields
  \[
     \deg_f(x) + 1 = \sum_{j=1}^k \tilde F_f(\vv_j).
  \]
  It follows that the crucial weight at $x$ is
  \begin{align*}
    \wR(x) &= \deg_f(x) - 1 + \Nshear(x) \\
    &= \sum_{j=1}^k \tilde F_f(\vv_j) - 2 + (r-k).
  \end{align*}
  Rearranging this expression gives the desired formula.
\end{proof}


\subsection{Trees associated with compressive domains}
\label{subsec:compressive_trees}

\noindent \textbf{Convention.} Throughout this section, we assume that a
compressive domain $U$ for $f$ has at least one boundary point. Equivalently,
$U \ne \BerkK$.

Our proof of Theorem~E requires a fine understanding of the location of crucial
points in a compressive domain $U$. In the proof of Theorem~E, we will show
that all of these crucial points lie on a particular subtree $\Gamma$ of $\bar
U$. Our goal for this section is to define that subtree and prove several
combinatorial results about it that are required for the crucial-weight count
in Subsection~\ref{subsec:proof_of_Theorem_E}.

For the purpose of this paper, we define a \textbf{finite tree} $\Gamma$ to be
the connected hull of a nonempty finite set of points of $\BerkK$. Note that
each point $x \in \Gamma$ has a natural notion of valence $v_\Gamma(x)$; it is
the number of tangent directions at $x$ that contain points of $\Gamma$. The
\textbf{leaves} of $\Gamma$ are the points with valence~1, and the
\textbf{branch points} of $\Gamma$ are the points with valence different
from~2. A \textbf{vertex set} for $\Gamma$ is any finite subset $V$ containing
all points with valence different from~2. The associated \textbf{edge set} $E$
consists of the maximal open intervals whose endpoints lie in $V$. Then $(V,E)$
determines a combinatorial tree structure on $\Gamma$, and $\Gamma$ is the
geometric realization of this combinatorial tree in the sense of Serre
\cite[Ch.1,\S2]{Serre_Trees_2003}.

Fix a compressive domain $U$, where $f$ has degree $d \ge 1$ and $U$ has $n \ge
1$ boundary points.  Let $x_1, x_2, \ldots, x_n$ be the boundary points of $U$,
let $y_1, y_2, \ldots, y_c$ be the distinct classical fixed points of $f$ in
$U$, and let $z_1, z_2, \ldots, z_m$ be the repelling fixed points in
$U$. Define
\[
  \Gamma_0 = \Gamma_0(f,U) = \Hull(x_1, \ldots, x_n, y_1, \ldots, y_c, z_1, \ldots, z_m).
\]

\begin{lemma}
  \label{lem:Gamma_non-trivial}
  The tree $\Gamma_0$ has at least two leaves.
\end{lemma}

\begin{proof}
  By Theorem~C, $U$ contains a classical fixed point, and by our standing
  assumption, $U$ has at least one boundary point.
\end{proof}

The following lemma will help us distinguish the leaves of $\Gamma_0$.

\begin{lemma}
  \label{lem:unique_fixed_no_shearing}
  A repelling fixed point $z_i$ is a leaf of $\Gamma_0$ if and only if $z_i$
  has a unique fixed direction and no shearing direction.
\end{lemma}

\begin{proof}
  Suppose that $z_i$ has a unique fixed direction $\vv$ and no shearing
  direction. All of the classical fixed points of $f$ must lie in
  $B_{z_i}(\vv)^{-}$. Each $x_j$ in $\partial U$ has a simply fixed direction
  pointing into $U$, and so each has another fixed direction pointing out of
  $U$. By Rumely's First Identification Lemma, there is a classical fixed point
  lying along every such direction. Thus, the points $x_1, \ldots, x_n, y_1,
  \ldots, y_c$ are all contained in $B_{z_i}(\vv)^{-}$. If $z_i$ is not a leaf
  of $\Gamma_0$, then there is a second direction $\ww \ne \vv$ such that
  $B_{z_i}(\ww)^-$ meets $\Gamma_0$. Then some $z_j$ lies in the direction
  $\ww$, where $j \ne i$. As no classical fixed point is contained in
  $B_{z_i}(\ww)^{-}$, Rumely's First Identification Lemma shows that the
  direction at $z_{j}$ toward $z_{i}$ is the unique fixed direction at
  $z_{j}$. By the Second Indifference lemma, there is a subarc $(z_j,t)$ of
  $(z_j, z_i]$ consisting entirely of id-indifferent fixed points.  But then
    Corollary~\ref{cor:bound_on_no_of_id_indiff_comp} shows that
    $B_{z_i}(\ww)^-$ contains a classical fixed point, which is a
    contradiction. We conclude that only one direction at $z_i$ meets
    $\Gamma_0$, and that $z_i$ is a leaf of $\Gamma_0$.

  Conversely, suppose that $z_{i}$ is a leaf of $\Gamma_{0}$. Let $\vv$ be the
  direction at $z_{i}$ into $\Gamma_{0}$. Let $\ww \ne \vv$ be another
  direction at $z_{i}$. As $z_{i}$ is a leaf of $\Gamma_{0}$,
  $B_{z_i}(\ww)^{-}$ is contained in $U$, and contains no classical fixed
  point. By Rumely's First Identification Lemma, $\ww$ is neither a fixed nor a
  shearing direction at $z_{i}$. As $\ww \ne \vv$ is chosen arbitrarily and
  $z_{i}$ must have at least one fixed direction, the only if part of the lemma
  is proved.
\end{proof}

Reordering the $z_i$'s if necessary, we may assume that $z_1, \ldots, z_e$ are
the repelling type~II fixed points in $U$ with a unique fixed direction and no
shearing direction, and that $z_{e+1}, \ldots, z_m$ each has either a shearing
direction or at least two fixed directions. It follows from
Lemma~\ref{lem:unique_fixed_no_shearing} that $\Gamma_0$ has precisely $n + c +
e$ leaves.

Now we modify the tree $\Gamma_0$ to obtain another tree $\Gamma = \Gamma(f,U)$
that is properly contained inside $U$ and has type~II boundary points. Choose
type~II points $x_1', \ldots, x_n', y_1', \ldots, y_c' \in \Gamma_0$ such that
\begin{itemize}
  \item $x_i'$ is close to $x_i$ and $\wR(x) = 0$ for all
   $x \in [x'_{i},x_{i})$, for $i=1, \ldots, n$, and 
  \item $y_i'$ is close to $y_i$ and $\wR(y) = 0$ for all
   $y \in [y'_{i}, y_{i})$ for $i=1, \ldots, c$.
\end{itemize}
To make precise what we mean by ``close'', let $V_{0}$ be a vertex set for
$\Gamma_{0}$ that contains all points with positive crucial weight. For each
$i$, we choose $x_i'$ on the unique edge that contains $x_i$ in its boundary,
and we choose $y_i'$ on the unique edge that contains $y_i$ in its boundary. In
the case where $U$ is a disk and $\Gamma_0$ has only two leaves, we further
insist that $x_1 < x_1' < y_1' < y_1$ after identifying $\Gamma_0 =
\Hull(x_1,y_1)$ with a real interval. As $U$ is a compressive domain, it
follows that $x'_i$ is nonfixed for all $i = 1, \ldots, n$. Define the finite
tree
\[
  \Gamma = \Gamma(f,U) = \Hull(x_1', \ldots, x_n', y_1', \ldots, y_c', z_1, \ldots, z_e).
\]
It follows that $\Gamma \subset \Gamma_0$ and that $v_{\Gamma_{0}}(x) =
v_\Gamma(x)$ for all $x \in \Gamma \smallsetminus \{x_1', \ldots, x_n', y_1',
\ldots, y_c'\}$. From Lemma~\ref{lem:Gamma_non-trivial}, $\Gamma$ has at least
two points.

Now we define a special vertex set for $\Gamma$. First, write
\[
    I_\Gamma = \{x \in \Gamma : f(x) = x \text{ and $x$ is id-indifferent}\} = I_f \cap \Gamma.
\]
As $\BerkK$ is uniquely path connected, the intersection of two connected
subsets of $\BerkK$ is connected, so
Corollary~\ref{cor:bound_on_no_of_id_indiff_comp} implies that $I_\Gamma$ has finitely
many connected components. In particular, $\partial I_\Gamma$ is a finite set. Define
a vertex set for $\Gamma$ by
\[
   V = \{x \in \Gamma : v_\Gamma(x) \ne 2\} \cup \partial I_\Gamma.
\]
Every point of $V$ is of type~II. We partition $V$ as follows:
\begin{align*}
    V_1 &= \{z_1, z_2, \ldots, z_e\} \\
    V_2 &= \{x \in V : x \text{ is a leaf of $\Gamma$ or an id-indifferent fixed point}\} \smallsetminus V_{1} \\
    V_3 &= \{x \in V : f(x) = x  \} \smallsetminus (V_1 \cup V_2) \\
    V_4 &= \{x \in V : f(x) \ne x\} \smallsetminus (V_1 \cup V_2) \\
\end{align*}

The following lemmas will be useful for computing the number of crucial points
in $\Gamma$.

\begin{lemma}
  \label{lem:classical_valence}
  For $x \in \Gamma \smallsetminus (V_2 \cup I_\Gamma)$, the number of directions at $x$
  containing a classical fixed point is $v_\Gamma(x)$.
\end{lemma}

\begin{proof}
  Suppose first that $x \in V_1$. By definition of the points in $V_1$, all
  classical fixed points lie in a single direction at $x$ and by
  Lemma~\ref{lem:unique_fixed_no_shearing}, $v_\Gamma(x) = 1$. In what remains,
  we assume that $x \in \Gamma \smallsetminus (V_1 \cup V_2 \cup I_\Gamma)$. In
  particular, $x$ is not a leaf of $\Gamma$.

  Let $\vv$ be a direction at $x$ that meets $\Gamma$. For the sake of a
  contradiction, assume that $B_x(\vv)^-$ contains no classical fixed
  point. Then no $x_i'$ and no $y_j'$ lies in $B_x(\vv)^-$. However, this ball
  must contain a leaf of $\Gamma$, so there is $z_i \in B_x(\vv)^-$ for some $i
  \le e$. By Lemma~\ref{lem:unique_fixed_no_shearing}, the direction at $z_i$
  that points back toward $x$ must be the unique fixed direction. By the Second
  Indifference Lemma, there is $t \in [x,z_i)$ such that the entire arc
    $(t,z_i)$ consists of id-indifferent fixed points. Since $x$ is not
    id-indifferent, the ball $B_x(\vv)^-$ contains a component of the
    id-indifference locus. By
    Corollary~\ref{cor:bound_on_no_of_id_indiff_comp}, the ball $B_x(\vv)^-$
    contains a classical fixed point, and we have reached a contradiction.

  Conversely, let $\vv$ be a direction at $x$ that contains a classical fixed
  point $y$. If $y \in U$, then $y = y_j$ for some $j$.  As $x$ is not a leaf,
  $y_j' \in B_x(\vv)^-$. If $y \not\in U$, then $B_x(\vv)^- \not\subset U$,
  which implies that there is a boundary point $x_i \in B_x(\vv)^-$. Again, as
  $x$ is not a leaf, we find $x_i' \in B_x(\vv)^-$, which concludes the proof.
\end{proof}

\begin{lemma}
  \label{lem:repelling_and_additive}
  Any additively indifferent or repelling fixed point in $\Gamma$ is contained
  in $V_1 \cup V_3$.
\end{lemma}

\begin{proof}
  Let $x \in \Gamma$ be an additively indifferent fixed point. Then $x \in
  \partial I_\Gamma \subset V$. Evidently, $x \not\in V_1 \cup V_4$. If $x$
  is a leaf of $\Gamma$, it must be of the form $y'_i$, for some $i$. However,
  every $y'_i$ has at least two directions containing a classical fixed point,
  namely one into $\Gamma$ and another towards $y_i$. Thus, $y'_i$ being 
  additively indifferent implies its crucial weight to be positive. This 
  contradicts the definition of $y'_i$. Hence, $x \in V_3$.  

  Now let $x \in \Gamma$ be a repelling fixed point. Then $x = z_j$ for some
  $j$. If $x$ is a leaf of $\Gamma$, then $x \in V_1$ by
  Lemma~\ref{lem:unique_fixed_no_shearing}.  Otherwise, $x \in V_3$.
\end{proof}

  Since $\BerkK$ is uniquely path-connected, the intersection of connected sets
  is connected. It follows that the components of $I_\Gamma$ are in bijective
  correspondence with the components of $I_f$ that meet $\Gamma$, which are
  also in bijective correspondence with the components of $I_f$ that meet
  $\Gamma_0$. For a component $I_0$ of $I_\Gamma$, let $J_0$ be the component
  of $I_f$ such that $J_0 \cap \Gamma = I_0$.

\begin{lemma}
  \label{lem:S3_helper}
  Let $I_0$ be a component of $I_\Gamma$. Write $u_1, \ldots, u_q$ for the
  type~II points on the boundary of $I_0$ that are fixed and repelling. For
  each $i = 1, \ldots, q$, write $\vv_i$ for the tangent direction at $u_i$
  that meets $I_0$. After reordering if necessary, write $y_{1}, \ldots, y_N$
  for the distinct classical fixed points contained in $\bar{J}_0.$ Write 
  $\ell(\bar I_0)$ for the number of leaves of the
  finite tree $\bar I_0$, and write $a_0$ for the number of additively
  indifferent fixed points on the boundary of $I_0$. Then 
  \[
  \sum_{i=1}^q \left(\tilde F_f(\vv_i) - 1\right)
  = \ell(\bar I_0) - a_0 - 2 + \sum_{i=1}^{N} ( F_f(y_i) - 1),
  \] 
  where $\tilde F_f(\vv_i)$ is the multiplicity of $\vv_i$ as a fixed point of
  $T_{u_i} f$ and $F_f(y_i)$ is the fixed-point multiplicity of $y_{i}$.
\end{lemma}

\begin{proof}
  Recall that $J_0$ is the component of $I_f$ such that $J_0 \cap \Gamma =
  I_0$. Note that $J_0 \subset U$ since the boundary points of $U$ are
  repelling.  As $\Gamma$ contains all the repelling fixed points in $U$, $u_1,
  \ldots, u_q$ are the only repelling points contained in $\bar J_0$. Then
  Proposition~\ref{prop:num_classically_indifferent} applied to $J_0$ shows
  that
  \begin{equation}
    \label{eq:I0_classical_count}
     \sum_{i=1}^{N} F_f(y_i) = 2 - q + \sum_{i=1}^q \Big(\tilde F_f(\vv_i) - 1\Big). 
  \end{equation}
  
  As the classical fixed points in $\bar J_0$ are $y_1, \ldots, y_N$, it
  follows that $y_i' \in \bar I_0$ if and only if $1 \le i \le N$. Let $x$ be a
  leaf of $\bar I_0$. Then either $x$ is an id-indifferent fixed point, and
  hence $x = y_i'$ for some $i \le N$, or else $x$ is a type~II fixed point
  such that $\tilde F_f(\vv) > 1$ for the direction $\vv$ at $x$ that meets
  $I_0$. In the latter case $x$ is either repelling or additively
  indifferent. It follows that
  \begin{equation}
    \label{eq:I0_leaf_count}
     \ell(\bar I_0) = q + a_0 + N.
  \end{equation}

  Combining \eqref{eq:I0_classical_count} with~\eqref{eq:I0_leaf_count} gives
  the desired formula.
\end{proof}

\begin{lemma}
  \label{lem:tree_leaves}
  Let $\Gamma$ be a finite tree with $\ell \ge 2$ leaves. Then
  \[
  \sum_{\substack{x \in \Gamma \\ v_\Gamma(x) \ge 2}} \left(v_\Gamma(x) - 2\right) = \ell - 2.
  \]
\end{lemma}
    
\begin{proof}
  As $\ell \ge 2$, $\Gamma$ has at least~2 points. Therefore, $v_\Gamma(x) \geq
  1$ for any point $x \in \Gamma$. Let $V$ be a set of vertices for $\Gamma$,
  and let $E$ be the associated set of edges. Since $\Gamma$ is a tree, we find
  that $\# V - \#E = 1$. The handshaking lemma from graph theory shows that
  \begin{equation}
    \label{eq:v_minus_two}
     \sum_{x \in \Gamma} \left(v_\Gamma(x) - 2\right) = 2 \# E - 2 \#V = -2. 
  \end{equation}
  Thus, 
  \[
  \sum_{\substack{x \in \Gamma \\ v_\Gamma(x) \ge 2}} \left(v_\Gamma(x) - 2\right)
  = -2 - \sum_{\substack{x \in \Gamma \\ v_\Gamma(x) = 1}} \left(v_\Gamma(x) - 2\right)
  = -2 + \ell. \qedhere
  \]
\end{proof}


\subsection{Proof of Theorem E}
  \label{subsec:proof_of_Theorem_E}

We are now ready to count the crucial weight contained in a compressive domain
$U$. Recall that we write $d \ge 1$ for the degree of $f$.

First, suppose that $U = \BerkK$. Note that this includes the case $d = 1$
since such a map has no repelling fixed point. The statement of Theorem~E is
precisely Rumely's Weight Formula \cite[Thm.~1.3]{Faber_Patra_fixed_structure},
since $f$ has precisely $d+1$ classical fixed points when counted with
multiplicity.

For the remainder of the argument, we assume that $U$ has at least~1 boundary
point. We keep the notation of the previous section, including the finite trees
$\Gamma_0, \Gamma$ and its distinguished vertex set $V = V_1 \cup V_2 \cup V_3
\cup V_4$. We begin by determining the crucial weight of the points in this
vertex set.

Suppose first that $x \in V_1 \cup V_3$, which means $x$ is a type~II fixed
point that is not id-indifferent. Combining
Proposition~\ref{prop:crucial_wt_formula} with
Lemma~\ref{lem:classical_valence}, we see
\[
\wR(x) = v_\Gamma(x) - 2 +
\sum_{\substack{\vv \in T_x \\ T_x f(\vv) = \vv}} \Big(\tilde F_f(\vv) - 1\Big)
\qquad (x \in V_1 \cup V_3).
\]
Next suppose that $x \in V_2$. By construction, the points $x_i'$ and $y_j'$
have trivial crucial weight, and by definition so does an id-indifferent
point. So $\wR(x) = 0$.  Finally, suppose that $x \in V_4$. Since $x$ is
non-fixed and not a leaf of $\Gamma$, the definition of the crucial weight
and Lemma~\ref{lem:classical_valence} shows that
\[
\wR(x) = v_\Gamma(x) - 2 \qquad (x \in V_4).
\]
  
From Proposition~\ref{prop:crucial_wt_formula} and Second Indifference Lemma, a
type~II point $x$ has positive crucial weight only if either it has three or
more directions containing a classical fixed point or $x \in \partial I_f$. By
definition of $\Gamma_0$, any point in $U \smallsetminus \Gamma_0$ has a unique
direction containing classical fixed points. Thus, a point in $U \smallsetminus
\Gamma_0$ has positive crucial weight only if it is a boundary point of
$I_f$. However, any such boundary point with positive weight, is either a
repelling point or an additively indifferent point with a shearing
direction. Either way, they lie in $\Gamma_0$ by construction.  As $\Gamma_0
\smallsetminus \Gamma$ contains no point of positive crucial weight, all points
in $U$ with positive crucial weight is contained in $\Gamma$. From the
discussion above and Lemma~\ref{lem:classical_valence}, for any point $x$ in
$\Gamma$ with positive crucial weight, the number of directions at $x$
containing a classical fixed point is $v_\Gamma (x)$. By definition of $V$, all
such points are contained in $V$.  Now, combining the formulas in the preceding
paragraph, we can estimate the total crucial weight in $U$:
\begin{align*}
  \sum_{x \in U} \wR(x) =  \sum_{x \in V} \wR(x)  
  &= \sum_{x \in V_1 \cup V_3} \left(v_\Gamma(x) - 2 + \sum_{\substack{\vv \in T_x \\ \text{fixed}}} \Big(\tilde F_f(\vv) - 1\Big)\right) + \sum_{x \in V_4} \left(v_\Gamma(x) - 2\right) \\
  &= \sum_{x \in V} \left(v_\Gamma(x) - 2\right) - \sum_{x \in V_2} \left(v_{\Gamma}(x) -2\right) + \sum_{x \in V_1 \cup V_3} \sum_{\substack{\vv \in T_x \\ \text{fixed}}} \Big(\tilde F_f(\vv) - 1\Big).
\end{align*}

Write $S_1, S_2$, and $S_3$ for the three sums over vertices in this final
line.

By Equation~\eqref{eq:v_minus_two}, we find that $S_1 = -2$.  Write $I_1,
\ldots, I_k$ for the connected components of $I_\Gamma = I_f \cap
\Gamma$. Since $\Gamma$ is nontrivial, by Third Indifference Lemma each closure
$\bar I_j$ is a finite tree with $\ell(\bar I_j) \ge 2$ leaves. Suppose that
$y_1, \ldots, y_c$ are the classical fixed points in $U$. Since $V_2$ consists
of the id-indifferent points of $V$ and all leaves of $\Gamma$ other than
$z_{1},z_{2}, \ldots, z_{e}$, we find that
\begin{align*}
  S_2 = \sum_{x \in V_2} \left(v_\Gamma(x) - 2\right) &=
  \sum_{\substack{\text{leaves of $\Gamma$,}\\ \text{distinct from }z_{i}}} (-1)
  + \sum_{\substack{\text{$x$ id-indifferent} \\ v_{\Gamma}(x) \ge 2}} \left(v_\Gamma(x) - 2\right) \\
  &= -n -c  + \sum_{j=1}^k \Big(\ell(\bar I_j) - 2\Big),
\end{align*}
where the final equality follows from Lemma~\ref{lem:tree_leaves} applied to
the finite trees $\bar I_1, \ldots, \bar I_k$.

To treat the sum $S_3$, we begin by defining $D_{r}$ (respectively, $D_{a}$) to
be the set of all pairs $(x, \vv)$ such that $x$ is a repelling fixed point
(respectively, additively indifferent fixed point) contained in $\Gamma$ and
$\vv \in T_{x}$ with $\tilde F_f(\vv) > 1$. For any $(x, \vv) \in D_{r} \cup
D_{a}$, the point $x \in V_1 \cup V_3$ by
Lemma~\ref{lem:repelling_and_additive}. Thus,
\begin{equation}
  \label{eq:S3_First}
     S_{3} = \sum_{x \in V_1 \cup V_3} \sum_{\substack{\vv \in T_x
         \\ \text{fixed}}} \Big(\tilde F_f(\vv) - 1\Big) = \sum_{(x, \vv) \in
       D_{r}} \Big(\tilde F_f(\vv) - 1\Big) + \sum_{(x, \vv) \in D_{a}}
     \Big(\tilde F_f(\vv) - 1\Big).
\end{equation}

Each pair $(x,\vv) \in D_r \cup D_a$ corresponds to a leaf of $\bar I_j$ for
some $j \in \{1, \ldots, k\}$, and each leaf of some $\bar I_j$ that is a
repelling or additively indifferent fixed point corresponds to a unique pair
$(x,\vv)$. So we may partition the terms of the sums according to their
associated component $\bar I_j$.  For any component $I_{j}$ of $I_{\Gamma}$,
let $N_{j}$ be the number of distinct classical fixed points contained in $\bar
J_j$, where $J_j$ is the component of $I_f$ such that $J_j \cap \Gamma =
I_j$. Write $y_{j,1}, \ldots, y_{j,N_j}$ for the distinct classical fixed
points in $\partial I_j$. Writing $a_j$ for the number of additively
indifferent fixed points in the boundary of $I_j$, Lemma~\ref{lem:S3_helper}
gives
  
  \begin{align*}
    S_3 &= \sum_{j=1}^k \sum_{(x,\vv) \in D_r \cap \partial I_j}  \Big(\tilde F_f(\vv) - 1\Big)
    + \sum_{j=1}^k \sum_{(x,\vv) \in D_a \cap \partial I_j}  \Big(\tilde F_f(\vv) - 1\Big) \\
    &= \sum_{j=1}^k \big(\ell(\bar I_j) - a_j - 2 + \sum_{i=1}^{N_j} ( F_f(y_{j,i}) - 1)\Big) + \sum_{j=1}^k a_j \\
    &= \sum_{j=1}^k \Big(\ell(\bar I_j) -2\Big) + \sum_{j=1}^{k} \sum_{i=1}^{N_j} ( F_f(y_{j,i}) - 1)\\
    &= \sum_{j=1}^k \Big(\ell(\bar I_j) -2\Big) + \sum_{i=1}^{c} ( F_f(y_i) - 1).
  \end{align*} 

The last equality follows from the fact that any classical fixed point in $U$
is either contained in $\bar I_{f}$ or its fixed point multiplicity is $1$. For
any classical fixed point $y$ with fixed point multiplicity $1$, the term
$F_f(y) - 1$ does not contribute to the sum.
  
To summarize, we have shown that
\begin{align*}
  S_1 &= -2 \\
  S_2 &= -n - c + \sum_{j=1}^k \left(\ell(\bar I_j) - 2\right) \\
  S_3 & = \sum_{j=1}^k \Big(\ell(\bar I_j) -2\Big) + \sum_{i=1}^{c} ( F_f(y_i) - 1).
\end{align*}
Therefore, 
\[
\sum_{x \in U} \wR(x) = S_1 - S_2 + S_3 = n + c - 2 + \sum_{i=1}^{c} ( F_f(y_i) - 1) = n + \bar c - 2. \qedhere
\]


\bibliographystyle{amsalpha}
\bibliography{fixed_locus}

@book{Baker-Rumely_BerkBook_2010,
  title = {Potential theory and dynamics on the {B}erkovich projective line},
  publisher = {American Mathematical Society},
  year = {2010},
  author = {Baker, Matthew and Rumely, Robert},
  volume = {159},
  pages = {xxxiv+428},
  series = {Mathematical Surveys and Monographs},
  address = {Providence, RI},
  isbn = {978-0-8218-4924-8},
  mrclass = {32H50 (14G20 14G22 20F65 31C15 31C45 37Bxx 37Pxx)},
  mrnumber = {MR2599526}
}

@ARTICLE{Chambert-Loir_Measures_2005,
  author = {Chambert-Loir, Antoine},
  title = {Mesures et \'equidistribution sur les espaces de {B}erkovich},
  journal = {J. Reine Angew. Math.},
  year = {2006},
  volume = {595},
  pages = {215--235},
  coden = {JRMAA8},
  fjournal = {Journal f\"ur die Reine und Angewandte Mathematik},
  issn = {0075-4102},
  mrclass = {14Gxx},
  mrnumber = {MR2244803}
}

@UNPUBLISHED{Faber_Patra_fixed_structure,
  author = {Faber, Xander and Patra, Niladri},
  title = {Structure of the Components of the Fixed Locus of a Self-Map of the
                  {B}erkovich Line},
  note = {preprint},
  year = {2026},
}

@ARTICLE{Favre_Rivera-Letelier_Ergodic_2010,
  author = {Favre, Charles and Rivera-Letelier, Juan},
  title = {Th\'eorie ergodique des fractions rationnelles sur un corps ultram\'etrique},
  journal = {Proc. Lond. Math. Soc. (3)},
  year = {2010},
  volume = {100},
  pages = {116--154},
  number = {1},
  doi = {10.1112/plms/pdp022},
  fjournal = {Proceedings of the London Mathematical Society. Third Series},
  issn = {0024-6115},
  mrclass = {37Pxx (11S85)},
  mrnumber = {2578470},
  url = {http://dx.doi.org/10.1112/plms/pdp022}
}

@book{Hatcher_AT,
  author    = {Hatcher, Allen},
  title     = {Algebraic Topology},
  publisher = {Cambridge University Press},
  address   = {Cambridge},
  year      = {2002},
}

@article{Jacobs_crucial_equidistribution,
  author  = {Jacobs, Kenneth},
  title   = {Equidistribution of the crucial measures in non-{A}rchimedean dynamics},
  journal = {Journal of Number Theory},
  volume  = {180},
  pages   = {86--138},
  year    = {2017},
  publisher = {Elsevier},
  issn    = {0022-314X},
  doi     = {10.1016/j.jnt.2017.02.018}
}

@unpublished{Patra_fixed,
  author = {Patra, Niladri},
  title = {Connected components of Berkovich fixed point locus}, 
  year = {2025},
  note = {arXiv:1104:0320 [math.DS], preprint}
}

@ARTICLE{Rumely_new_equivariant,
  author = {Rumely, Robert},
  title = {A New Equivariant in Nonarchimedean Dynamics},
  journal = {Algebra Number Theory},
  year = {2017},
  volume = {11},
  issue = {4},
  pages = {841--844}
}

@book{Serre_Trees_2003,
  author     = {Serre, Jean-Pierre},
  title      = {Trees},
  series     = {Springer Monographs in Mathematics},
  translator = {Stilwell, John},
  publisher  = {Springer-Verlag},
  address    = {Berlin},
  year       = {2003},
  edition    = {Corrected 2nd printing},
  pages      = {ix + 142},
  isbn       = {978-3-540-44237-0},
  doi        = {10.1007/978-3-642-61856-7},
  note       = {Translation of \emph{Arbres, amalgames, $SL_2$},
                Soc.\ Math.\ France, 1977}
}

\end{document}